\documentclass[a4paper, reqno, 12pt]{amsart}
\usepackage[utf8]{inputenc} 
\usepackage[T1, T2A]{fontenc}
\usepackage[english]{babel}
\usepackage{amsmath, amssymb, amsfonts, amsthm, amscd, mathrsfs,stmaryrd}
\usepackage{enumitem}
\usepackage[hidelinks]{hyperref}
\usepackage[usenames]{color}
\setlist[enumerate]{label = (\alph*), ref=(\text{\alph*)}}
\setlist[itemize]{nolistsep}

\usepackage{epsfig}
\usepackage{graphicx}
\usepackage[cmtip,arrow]{xy}
\usepackage{bm}
\usepackage{tikz}

\renewcommand{\phi}{\varphi}
\renewcommand{\ge}{\geqslant}
\renewcommand{\le}{\leqslant}

\newcommand{\FF}{\mathbb{F}}
\newcommand{\RR}{\mathbb{R}}
\newcommand{\KK}{\mathbb{K}}
\renewcommand{\AA}{\mathbb{A}}
\newcommand{\ZZ}{\mathbb{Z}}
\newcommand{\GG}{\mathbb{G}}
\newcommand{\PP}{\mathbb{P}}
\def\Xf{{\mathcal{X}}}
\def\Zf{{\mathcal{Z}}}
\newcommand{\fX}{\mathfrak{X}}
\newcommand{\NN}{\mathbb{Z}_{>0}}
\newcommand{\Zgezero}{\mathbb{Z}_{\geqslant 0}}
\newcommand{\pa}{\partial}
\newcommand{\Xreg}{X^{\mathrm{reg}}}
\newcommand{\yu}{{\bf y}}

\DeclareMathOperator{\SL}{SL}
\DeclareMathOperator{\Aut}{Aut}
\DeclareMathOperator{\SAut}{SAut}
\DeclareMathOperator{\rk}{rk}
\DeclareMathOperator{\Pic}{Pic}
\DeclareMathOperator{\Spec}{Spec}
\DeclareMathOperator{\diag}{diag}
\DeclareMathOperator{\Susp}{Susp}

\theoremstyle{plain}
\newtheorem{lemma}{Lemma}
\newtheorem{proposition}{Proposition}
\newtheorem{theorem}{Theorem}
\newtheorem{corollary}{Corollary}
\theoremstyle{definition}
\newtheorem{definition}{Definition}
\newtheorem{example}{Example}

\theoremstyle{remark}
\newtheorem{remark}{Remark}

\begin{document}

\title[Affine homogeneous varieties and suspensions]{Affine homogeneous varieties and suspensions}
\author{Ivan Arzhantsev}
\address{HSE University, Faculty of Computer Science, Pokrovsky Boulvard 11, Moscow, 109028 Russia}
\email{arjantsev@hse.ru}

\author{Yulia Zaitseva}
\address{HSE University, Faculty of Computer Science, Pokrovsky Boulvard 11, Moscow, 109028 Russia}
\email{yuliazaitseva@gmail.com}

\thanks{This research was supported by the Ministry of Science and Higher Education of the Russian Federation, agreement 075-15-2022-289 date 06/04/2022}

\subjclass[2010]{Primary 14M17, 14R20; \ Secondary 14J50, 14L30}

\keywords{Affine algebraic variety, homogeneous space, automorphism group, transitivity, suspension, Danielewski surface, Picard group}

\begin{abstract}
An algebraic variety $X$ is called a homogeneous variety if the automorphism group $\Aut(X)$ acts on~$X$ transitively, and a homogeneous space if there exists a transitive action of an algebraic group on $X$. We prove a criterion of smoothness of a suspension to construct a wide class of homogeneous varieties. As an application, we give criteria for a Danielewski surface to be a homogeneous variety and a homogeneous space. Also, we construct affine suspensions of arbitrary dimension that are homogeneous varieties but not homogeneous spaces. 
\end{abstract}

\maketitle

\section{Introduction}

Let $X$ be an algebraic variety over an algebraically closed field~$\KK$ of characteristic zero and $G$ be a group acting on~$X$. Recall that the action of $G$ on~$X$ is \emph{transitive} if for any points $x, y \in X$ there is an element $g \in G$ such that $gx=y$. The variety $X$ is called a \emph{homogeneous variety} if the automorphism group $\Aut(X)$ acts on $X$ transitively. The variety $X$ is called a \emph{homogeneous space} if there exists a transitive action of an algebraic group $G$ on~$X$. In this case $X$ can be identified with the variety of left cosets $G/H$, where $H$ is the stabilizer in~$G$ of a point in~$X$. Homogeneous spaces are classical mathematical objects with rich structural theory and many applications; see e.g.~\cite{Br2017, BSU2013, Gr1997, Hu1975, OV, PV1994, Ti2011}. 

Clearly, any homogeneous space is a homogeneous variety. In general case $\Aut(X)$ is not an algebraic group. Moreover, there exist homogeneous varieties that are not homogeneous spaces. Any smooth quasi-affine toric variety is homogeneous, see~\cite[Theorem~0.2(2)]{AKZ2012} and~\cite[Theorem~4.3(a)]{ArShZa2022}. In~\cite[Example~2.2]{AKZ2012}, an example of a quasi-affine toric variety that is a homogeneous variety but not a homogeneous space is given, and in~\cite{ArShZa2022} the question is raised whether there exists an example in the class of affine varieties. There is no such affine toric variety since any homogeneous variety is smooth, and any smooth affine toric variety is isomorphic to the direct product $(\KK)^s \times (\KK^\times)^r$, $s, r \in \Zgezero$, which is a homogeneous space. In this paper we give series of affine surfaces and higher-dimensional varieties that are homogeneous varieties but not homogeneous spaces. For this purpose we use the construction of suspension. 

In Section~\ref{susp_sec} we notice that any smooth suspension over a flexible affine variety is a homogeneous variety and study when a suspension is smooth. Namely, the suspension $\Susp(Y, f) = \{uv = f(y)\} \subseteq \AA^2 \times Y$ over an affine variety $Y$ with $f \in \KK[Y]\setminus\KK$ is proved to be smooth if and only if the variety $Y$ and the scheme $\Spec \KK[Y]/(f)$ are smooth. This gives a criterion of smoothness of iterated suspensions and a construction of a class of homogeneous varieties. 

In Section~\ref{dim2_sec} we prove that the Danielewski surface $xz^n = f(y)$ is a homogeneous variety but not a homogeneous space if and only if $n=1$, the polynomial~$f$ has no multiple roots and $\deg f \ge 3$. The condition $n=1$ means that the Danielewski surface is a suspension over the affine line. 

We conclude Section~\ref{dim2_sec} with a discussion of Gizatullin and Danilov-Gizatullin surfaces. This leads to one more infinite series of affine homogeneous surfaces that are not homogeneous spaces.

In Section~\ref{dimn_sec} we obtain an upper bound on the rank of the Picard group of an affine homogeneous space and calculate the Picard group of some smooth affine suspensions. This allows to provide a family of affine varieties of arbitrary dimension that are homogeneous varieties but not homogeneous spaces. 

The work was carried out within the contest of mathematical projects of the Euler International Mathematical Institute. 

\section{Regular suspensions}
\label{susp_sec}

Let $\KK$ be an algebraically closed field of characteristic zero and $\GG_\mathrm{a} = (\KK, +)$ be the additive group of the ground field~$\KK$. Let $\GG_\mathrm{a} \times X \to X$ be a regular action on an algebraic variety $X$. The corresponding subgroup of the automorphism group $\Aut(X)$ is called a \emph{$\GG_\mathrm{a}$-subgroup} in $\Aut(X)$. By the \emph{special automorphism group} $\SAut(X)$ of a variety $X$ we call the subgroup of the automorphism group $\Aut(X)$ generated by all $\GG_\mathrm{a}$-subgroups in~$\Aut(X)$. Denote by $\Xreg$ the set of smooth points in~$X$. 

\begin{definition}
A smooth point $x$ of a variety $X$ is called \emph{flexible} if the tangent space to~$X$ at the point~$x$ is generated by tangents to orbits of $\GG_\mathrm{a}$-subgroups passing through the point~$x$. A variety $X$ is called \emph{flexible} if any smooth point $x \in \Xreg$ is flexible. 
\end{definition}

In \cite[Theorem 0.1]{AFKKZ12013}, it is proved that the following conditions are equivalent for an irreducible affine variety $X$: 
\begin{enumerate}
  \item the variety $X$ is flexible;
  \item the group $\SAut(X)$ acts on $\Xreg$ transitively.
\end{enumerate}
Moreover, if the variety $X$ has dimension at least~$2$, then these conditions are equivalent to 
\begin{enumerate}
  \item[(c)] the group $\SAut(X)$ acts on $\Xreg$ infinitely transitive.
\end{enumerate}
Recall that an action of a group $G$ on a set $S$ is called infinitely transitive if it is $m$-transitive for any $m \in \NN$, i.e. for any pairwise distinct points $x_1, \ldots, x_m \in S$ and any pairwise distinct points $y_1, \ldots, y_m \in S$ there exists an element $g \in G$ such that $gx_i = y_i$ for all $1 \le i \le m$. 

Now we recall the notion of suspension. In the context of automorphism groups suspensions were considered for the first time in~\cite{KaZa1999}. 

\begin{definition}
Let $Y$ be an affine variety and $f \in \KK[Y]$ be a nonconstant regular function on~$Y$. Then the hypersurface $\Susp(Y, f)$ that is given in the direct product $\AA^2 \times Y$ by the equation $uv = f(y)$, where $y \in Y$ and $\AA^2 = \Spec \KK[u,v]$, is called a \emph{suspension} over $Y$.
\end{definition}

We are interested in suspensions since this construction produces flexible varieties from flexible varieties~\cite{AKZ2012}; for the case of the ground field~$\RR$, see~\cite{KuMa2012}. Let us formulate the corresponding result. 

\begin{theorem}[{\cite[Theorem 0.2(3)]{AKZ2012}}]
\label{AKZ2012_023theor}
Suppose that an irreducible affine variety $X$ of positive dimension is flexible. Then any suspension over~$X$ is flexible as well. 
\end{theorem}

So we come to the following source of homogeneous varieties. 

\begin{proposition} \label{homog_prop}
Any smooth suspension over a flexible irreducible affine variety is a homogeneous irreducible affine variety. 
\end{proposition}

\begin{proof}
Let $X = \Susp(Y, f)$ be a smooth suspension over a flexible irreducible affine variety~$Y$. Since $Y$ is irreducible, the variety $X$ is irreducible as well according to~\cite[Lemma~3.1]{AKZ2012}. By Theorem~\ref{AKZ2012_023theor} and \cite[Theorem 0.1]{AFKKZ12013}, the subgroup $\SAut(X) \subseteq \Aut(X)$ acts on~$\Xreg$ transitively. Thus, the variety~$X$ is homogeneous since $X = \Xreg$. 
\end{proof}

To apply Proposition~\ref{homog_prop}, we have to investigate whether a suspension is smooth. First let us fix some definitions and notations. 

\medskip

Suppose $X \subseteq \AA^s$ is an affine variety and $I(X) \subseteq \KK[\AA^s] = \KK[T_1, \ldots, T_s]$ is the ideal of all polynomials that are zero on~$X$. Let $I(X) = (F_1, \ldots, F_t)$. Consider the Jacobian matrix
\[J_X(x) = \begin{pmatrix}\frac{\pa F_1}{\pa T_1}(x) & \ldots & \frac{\pa F_1}{\pa T_s}(x) \\ \ldots & \ldots & \ldots \\ \frac{\pa F_t}{\pa T_1}(x) & \ldots & \frac{\pa F_t}{\pa T_s}(x)\end{pmatrix}, \; x \in X.\]
It is known that the tangent space $T_xX$ to the variety~$X$ at a point~$x$ can be identified with the kernel of the linear map defined by the matrix~$J_X(x)$, see e.g.~\cite[Remark 5.2]{Kr1984}. In particular, 
\begin{equation}
\label{dim_Jac_eq}
\dim T_xX = s - \rk J_X(x). 
\end{equation}

It can be generalized to schemes, see~\cite[Section~13.1.7]{Va2023}. Namely, let $\Xf = \Spec \KK[T_1, \ldots, T_s] / (F_1, \ldots, F_t)$ be an affine $\KK$-scheme. In the reduced sense, $\Xf$ can be thought as a subvariety of~$\AA^s$ defined by equations $F_1 = \ldots = F_t = 0$. Then the tangent space $T_x\Xf$ to the scheme~$\Xf$ at a closed point~$x$ is the kernel of the linear map defined by the matrix~$J_\Xf(x)$ and equation~\eqref{dim_Jac_eq} holds as well. 

A variety (a scheme) is smooth at a point if the dimension of the variety (the scheme) coincides with the dimension of the tangent space at this point, and singular otherwise. 

\begin{lemma}
\label{susp_ideal_lem}
Let $Y \subseteq \AA^d$ be an affine variety and $f \in \KK[Y]$ be a nonconstant regular function on~$Y$. Consider the suspension $X = \Susp(Y, f) \subseteq \AA^{d+2}$. Let $I(Y) = (f_1, \ldots, f_m)$, $f_i \in \KK[\AA^d]$. Then for any polynomial $\widehat f \in \KK[\AA^d]$ which restricts to the function~$f$ on~$Y$ we have
\[I(X) = (uv - \widehat f, f_1, \ldots, f_m) \subseteq \KK[\AA^{d+2}].\]
\end{lemma}

\begin{proof}
By definition, $X$ is the zero set of the ideal $I = (uv - \widehat f, f_1, \ldots, f_m)$ in the polynomial algebra $\KK[\AA^{d+2}] = \KK[u, v, \yu]$, where $\yu = (y_1, \ldots, y_d)$. We have to check that $I$ is radical. 

Let $g^k \in I$ for some $g \in \KK[u, v, \yu]$. 
Consider the lexicographic order on $\KK[u, v, \yu]$ with $u \succ v \succ y_1 \succ \ldots \succ y_d$. The division of the polynomial $g$ by the polynomial $uv - \widehat f$ with respect to this order gives a decomposition 
\begin{equation} \label{Groebdiv_eq1}
g = (uv - \widehat f) h + r,
\end{equation}
where $h, r \in \KK[u, v, \yu]$ and $r$ is a polynomial that has no terms divisible by the leading term $uv$ of the polynomial $uv - \widehat f$. 

Substitute in $\yu$ coordinates of a point $y \in Y$ in equation~\eqref{Groebdiv_eq1}. We obtain 
\begin{equation} \label{Groebdiv_eq2}
g(y) = (uv - \widehat f(y)) h(y) + r(y),
\end{equation}
where $g(y), h(y), r(y) \in \KK[u,v]$, $\widehat f(y) \in \KK$. Still the remainder $r(y)$ has no terms divisible by $uv$. Since $g^k \in I = (uv - \widehat f, f_1, \ldots, f_m) \subseteq \KK[u,v,\yu]$ and $f_1(y) = \ldots = f_m(y) = 0$, it follows that 
\begin{equation} \label{Groebdiv_eq3}
g(y)^k \in (uv - \widehat f(y)) \subseteq \KK[u,v].
\end{equation}
The polynomial $uv - \widehat f(y) \in \KK[u,v]$ is square-free, so from equation~\eqref{Groebdiv_eq3} it follows that $g(y) \in (uv - \widehat f(y))$ as well. Then according to~\eqref{Groebdiv_eq2} we also have $r(y) \in (uv - \widehat f(y))$. Since $r(y)$ has no terms divisible by $uv$, it follows that $r(y) = 0$. Since $y \in Y$ is an arbitrary point of~$Y$, we obtain $r \in I(Y) = (f_1, \ldots, f_m)$. Then by equation~\eqref{Groebdiv_eq1} we have $g \in I$. Thus, the ideal $I$ is radical. 
\end{proof}

\begin{theorem}
\label{sm_susp_theor}
Let $Y$ be an affine variety and $f \in \KK[Y]$ be a nonconstant regular function on~$Y$. Then the suspension $X = \Susp(Y, f)$ is singular at a point $x=(u,v,y)\in X$ if and only if at least one of the following conditions holds: 
\begin{enumerate}
  \item the variety $Y$ is singular at the point $y \in Y$;
  \item $u = v = 0$ and the scheme $\Zf = \Spec \KK[Y]/(f)$ is singular at the point $y \in \Zf$.
\end{enumerate}
\end{theorem}

\begin{remark}
In the reduced sense $\Zf$ can be thought as the subvariety $\{f = 0\} \subseteq Y$. 
\end{remark}

\begin{remark}
In~\cite[Lemma 3.2]{AKZ2012}, it is proved that if $\pi\colon X \to Y$ is the restriction of the projection $\AA^2 \times Y \to Y$ then $\pi(\Xreg)= Y^{\mathrm{reg}}$. 
\end{remark}

\begin{proof}
Let $Y \subseteq \AA^d$ and $\KK[\AA^d] = \KK[\yu]$, where $\yu = (y_1, \ldots, y_d)$. Let $I(Y) = (f_1,\ldots, f_m)$, $f_i \in \KK[\yu]$, and a polynomial $\widehat f \in \KK[\yu]$ restricts to the function~$f$ on~$Y$. By Lemma~\ref{susp_ideal_lem}, 
\[I(X) = (uv-\widehat f(\yu), f_1(\yu), \ldots, f_m(\yu)) \subseteq \KK[u,v,\yu].\]
Let us calculate the Jacobian matrix at a point $x = (u,v,y) \in X$: 
\[
J_X(x) = \begin{pmatrix}
v & u & -\frac{\pa \widehat f}{\pa y_1}(y) & \ldots & -\frac{\pa \widehat f}{\pa y_d}(y)\\
0 & 0 & \frac{\pa f_1}{\pa y_1}(y) & \ldots & \frac{\pa f_1}{\pa y_d}(y)\\
\ldots & \ldots & \ldots & \ldots & \ldots \\
0 & 0 & \frac{\pa f_m}{\pa y_1}(y) & \ldots & \frac{\pa f_m}{\pa y_d}(y)
\end{pmatrix}
\]
Notice that the deleting of the first two columns and the first row gives the Jacobian matrix of~$Y$ at~$y$. Then
$\rk J_X(x) \le \rk J_Y(y) + 1$. For the dimensions of tangent spaces it means 
$d + 2 - \dim T_xX \le d - \dim T_yY + 1$, i.e. 
\begin{equation}
\label{dimTXTY_eq}
\dim T_xX - 1 \ge \dim T_yY.
\end{equation}
We also know that
\begin{equation}
\label{dimTYY_eq}
\dim T_yY \ge \dim Y = \dim X - 1.
\end{equation}
It follows that $\dim T_xX = \dim X$ if and only if both inequalities in~\eqref{dimTXTY_eq} and~\eqref{dimTYY_eq} turn into equalities. So the condition that the variety~$X$ is singular at a point~$x$, i.e. $\dim T_xX \ne \dim X$, is equivalent to one of the following:
\begin{enumerate}
  \item the inequality in~\eqref{dimTYY_eq} is strict, i.e. $Y$ is singular at the point~$y$;
  \item $Y$ is smooth at $y$ but inequality~\eqref{dimTXTY_eq} is strict, i.e. $\rk J_X(x) = \rk J_Y(y)$. 
\end{enumerate}

Consider the second case. The first row of the Jacobian matrix $J_X(x)$ has to be a linear combination of other rows. It follows that $u = v = 0$. Since $uv = \widehat f(y)$ we obtain that $y \in \Zf$. Notice that the deleting of the first two columns of $J_X(x)$ gives the Jacobian matrix of $\Zf$ in $y$. Since $\rk J_X(x)$ and $\rk J_Y(y)$ coincide, they also equal $\rk J_\Zf(y)$. Then $\dim T_y \Zf = \dim T_y Y = \dim Y$, but $\dim \Zf = \dim Y - 1$. Thus, $\Zf$ is singular at the point $y \in \Zf$ in the second case. 
\end{proof}

This result motivates the following definition. 
\begin{definition}
Let $Y$ be an affine variety and $f \in \KK[Y]$ be a nonconstant regular function on~$Y$. The suspension $\Susp(Y, f)$ is called \emph{regular} if the variety $Y$ and the scheme ${\Zf = \Spec \KK[Y]/(f)}$ are smooth. 
\end{definition}

We obtain several corollaries of Theorem~\ref{sm_susp_theor}. 

\begin{corollary}
\label{reg_susp_cor}
Let $Y$ be an affine variety and $f \in \KK[Y]$ be a nonconstant regular function on~$Y$. The suspension $\Susp(Y, f)$ is smooth if and only if it is regular. 
\end{corollary}

Now we can consider an \emph{iterated suspension} over an affine variety~$Y$, i.e. the variety $Y_k$ obtained by a series of suspensions
\begin{equation}\begin{aligned} \label{iter_susp_eq}
Y_1 &=\, Y,\\
Y_2 &=\, \Susp(Y_1, f_1),\\
&\ldots\\
Y_k &=\, \Susp(Y_{k-1}, f_{k-1}),
\end{aligned}\end{equation}
where $f_i \in \KK[Y_i] \setminus \KK$, $1 \le i \le k-1$. 

\begin{corollary}
\label{iter_susp_smooth_cor}
Let $Y_k$ be an iterated suspension over an affine variety $Y$. Then $Y_k$ is smooth if and only if the variety $Y$ and the schemes $\Zf_i = \Spec \KK[Y_i]/(f_i)$, $1 \le i \le k-1$, are smooth. 
\end{corollary}

\begin{proof}
By Corollary~\ref{reg_susp_cor}, $Y_k$ is smooth if and only if any suspension in~\eqref{iter_susp_eq} is regular. 
\end{proof}

The following corollary gives a wide class of affine homogeneous varieties. 

\begin{corollary}
\label{iter_susp_homog_cor}
Let $Y$ be a smooth flexible irreducible affine variety of positive dimension and~$Y_k$ be an iterated suspension over $Y$. Then $Y_k$ is a homogeneous variety if and only if the schemes $\Zf_i = \Spec \KK[Y_i]/(f_i)$, $1 \le i \le k-1$, are smooth. 
\end{corollary}

\begin{proof}
According to Corollary~\ref{iter_susp_smooth_cor}, the variety $Y_k$ is smooth if and only if the schemes $\Zf_i$, ${1 \le i \le k-1}$, are smooth. Any homogeneous variety is smooth, so it is a necessary condition. Conversely, if $Y$ is smooth, then by Proposition~\ref{homog_prop} the variety $Y_k$ is homogeneous. 
\end{proof}

\begin{example}
Suspensions over $\SL_2(\KK)$ can be considered as iterated suspensions over~$\AA^2$. For example, the suspension $Y_3 = \Susp(Y_2, x^3z+yt^2)$ over the variety \[Y_2 = \SL_2(\KK) = \left\{\det\begin{pmatrix}x & z\\t & y\end{pmatrix} = 1\right\}\] is the iterated suspension over $Y_1 = \AA^2 = \Spec \KK[x,y]$ given by the system of equations
$$\left\{\begin{aligned}
zt &= xy - 1\\
uv & = x^3z+yt^2
\end{aligned}\right.$$
Let us apply Corollary~\ref{iter_susp_homog_cor} to prove that the variety~$Y_3$ is homogeneous. The hyperbola $\Zf_1 = \Spec \KK[x,y]/(xy-1)$ is smooth. 
Consider \[\Zf_2 = \Spec\KK[\SL_2(\KK)] / (x^3z+yt^2) = \Spec\KK[z,t,x,y] / (zt - xy + 1, \, x^3z+yt^2).\] 
If the Jacobian matrix of $\Zf_2$
\[J = \begin{pmatrix} t & z & -y & -x\\ x^3 & 2yt & 3x^2z & t^2\end{pmatrix}\] at a point of $\Zf_2$ has rank~$1$, then $2yt^2 - x^3z = 0$. Taking into account $x^3z+yt^2 = 0$ we obtain $x^3z = yt^2 = 0$, whence according to $zt - xy + 1 = 0$ we see that either $x = y = 0$, $z, t \ne 0$, or $z = t = 0$, $x, y \ne 0$. This contradicts the condition $t^3+x^4=0$, which follows from $\rk J = 1$. 
\end{example}

\pagebreak[2]

Let us consider the case $Y = \AA^n$. 

\begin{corollary}
\label{susp_An_cor}
Let $X = \Susp(\AA^n, f)$ and $f = p_1\ldots p_k \in \KK[y_1, \ldots, y_n]$, where $p_i$ are irreducible polynomials. Then $X$ is homogeneous if and only if $p_i \ne p_j$ for $i \ne j$ and the subvarieties $\{p_i = 0\} \subseteq \AA^n$ are smooth and do not intersect pairwise.
\end{corollary}

\begin{proof}
According to Corollary~\ref{iter_susp_homog_cor}, the suspension~$X$ is homogeneous if and only if the scheme $\Spec \KK[\AA^n]/(f)$ is smooth. 

A closed point~$y$ of $\Spec \KK[\AA^n]/(f)$ is singular if and only if $\frac{\pa f}{\pa y_1}(y) = \ldots = \frac{\pa f}{\pa y_n}(y) = 0$. We may assume, without loss of generality, that $p_1(y) = 0$. Since \[\frac{\pa f}{\pa y_i} = \frac{\pa p_1}{\pa y_i}p_2 \ldots p_k + \ldots + p_1\ldots p_{k-1}\frac{\pa p_k}{\pa y_i},\]
the condition of singularity is equivalent to~$\frac{\pa p_1}{\pa y_i}(y)p_2(y) \ldots p_k(y) = 0$ for any $1 \le i \le n$, i.e. the point $y$ belongs to some $\{p_j = 0\}$ for $j \ne i$ or $y$ is a singular point of $\{p_1 = 0\}$. 
\end{proof}

\begin{example}
The affine hypersurface
\[
\{x^2y + x + uv = 0\} \subseteq \AA^4
\]
is homogeneous according to Corollary~\ref{susp_An_cor}. It is the suspension $\Susp(\AA^2, f)$, where the polynomial $f(x,y) = -x(xy+1)$ has simple irreducible factors, and the subvariety ${\{f = 0\}} \subseteq \AA^2$ has smooth irreducible components $\{x=0\}$ and $\{xy=-1\}$ that do not intersect. 
\end{example}

\begin{remark}
The condition on irreducible components in Corollary~\ref{susp_An_cor} can be checked via computer algebra methods. Namely, for $p_1,p_2 \in \KK[y_1, \ldots, y_n]$ the hypersurfaces $\{p_1 = 0\}$ and $\{p_2 = 0\}$ do not intersect if and only if the Gr\"obner basis of the ideal $(p_1, p_2)$ contains a constant. Alternatively, this condition holds if and only if the resultant of~$p_1$ and $p_2$ is a nonzero constant, where $p_1, p_2$ are considered as polynomials in $y_1$ with coefficients in $\KK[y_2,\ldots, y_n]$. The component~$\{p = 0\}$ is smooth if and only if the Gr\"obner basis of the ideal $(p, \frac{\pa p}{\pa y_1}, \ldots, \frac{\pa p}{\pa y_n})$ contains a constant. 
\end{remark}

\section{The case of surfaces}
\label{dim2_sec}

In this section we provide a series of explicit examples of affine surfaces that are homogeneous varieties but not homogeneous spaces. 

Let $x, y, z$ be coordinates in $\AA^3$. A \emph{Danielewski surface} is a surface in $\AA^3$ given by equation $xz^n = f(y)$, where $n \in \NN$ and $f \in \KK[y]$. It is known that two Danielewski surfaces with parameters $n_1, n_2 \in \NN$ and polynomials $f_1(y), f_2(y)$ are isomorphic if and only if $n_1 = n_2$ and 
\begin{equation}
\label{Daniel_change_eq}
f_1(y) = af_2(by+c)
\end{equation}
for some $a,b \in \KK^\times$, $c \in \KK$, see \cite[Lemma 2.10]{Da2004}. 

\smallskip

Let us formulate the result. 

\begin{theorem}
\label{Dan_theor}
Let $X$ be an affine surface given in $\AA^3$ with coordinates $x, y, z$ by equation $xz^n = f(y)$, where $f$ is a nonconstant polynomial. Then 
\begin{enumerate}
  \item[(i)] $X$ is a homogeneous variety if and only if $\deg f = 1$ or $n=1$ and the polynomial~$f$ has no multiple roots;
  \item[(ii)] $X$ is a homogeneous space if and only if $\deg f = 1$ or $n=1$, the polynomial~$f$ has no multiple roots and $\deg f = 2$. 
\end{enumerate}
\end{theorem}

\begin{remark}
Statement (i) of Theorem~\ref{Dan_theor} is known, but for convenience of the reader we recall the corresponding references in the proof. 
\end{remark}

We use the following statement to prove Theorem~\ref{Dan_theor}(ii). An algebraic variety is called \emph{quasihomogeneous with respect to an algebraic group} if it admits an action of an algebraic group with an open orbit such that the complement to this orbit is finite. In \cite{Gi1971_alggr}, all smooth irreducible quasihomogeneous with respect to an algebraic group affine surfaces are found. They are the following: 
\begin{enumerate}
  \item \label{surf_a} $\AA^2$;
  \item \label{surf_b} $\AA^1 \times \KK^\times$, where $\KK^\times = \KK \setminus \{0\}$;
  \item \label{surf_v} $(\KK^\times)^2$;
  \item \label{surf_g} $X_1 = \PP^2 \setminus C$, where $C$ is a smooth quadric;
  \item \label{surf_d} $X_2 = \PP^1 \times \PP^1 \setminus \diag \PP^1$, where $\diag$ is the diagonal in the direct product.
\end{enumerate}
In \cite{Po1973}, both smooth and singular irreducible quasihomogeneous with respect to an algebraic group affine surfaces are classified. It is known that the surface $X_1$ in item~\ref{surf_g} is isomorphic to the homogeneous space $\SL_2/N$, and a surface $X_2$ in item~\ref{surf_d} is isomorphic to the homogeneous space $\SL_2/T$, where $T$ is the one-dimensional subgroup of diagonal matrices in $\SL_2$ and $N$ is the normalizer of the torus $T$, see \cite[Lemma~2]{Po1973}. 

\begin{proof}[{Proof of Theorem~\ref{Dan_theor}}] 
If $\deg f = 1$ then $X$ is isomorphic to the affine plane~$\AA^2$ and is a homogeneous space. Hereafter $\deg f \ge 2$. 

(i) In~\cite{ML2001}, generators of the automorphism group of $X$ for $n > 1$ are found:
\begin{gather*}
\xi_\lambda\colon\;\; x \mapsto \lambda^{-1}x, \quad y \mapsto y, \quad z \mapsto \lambda z, \quad\text{ where } \lambda \in \KK^\times; \label{ML_eq1}\\
\theta_q\colon\;\; x \mapsto x + \frac{f(y+z^nq(z)) - f(y)}{z^n}, \quad y \mapsto y + z^nq(z), \quad z \mapsto z, \quad\text{ where } q(z) \in \KK[z]; \label{ML_eq2}\\
x \mapsto \lambda^dx, \quad y \mapsto \lambda y, \quad z \mapsto z, \quad\text{ where } \lambda \in \KK^\times, \quad\text{ if } f(y)=y^d;\\
x \mapsto \mu^dx, \quad y \mapsto \mu y, \quad z \mapsto z, \quad\text{ where }\mu^m = 1, \quad\text{ if } f(y)=y^d p(y^m), \; p(y) \in \KK[y]. 
\end{gather*}
One can notice that the set of points with $z=0$ in $X$ is invariant with respect to $\Aut(X)$, so $X$ is not homogeneous. Note that a generic orbit of the group $\SAut(X)$ is one-dimensional, and the group $\Aut(X)$ has an open orbit $\{z \ne 0\}$. Indeed, automorphisms~$\theta_q$ with $q(z) \in \KK^\times$ connect points with fixed $z \ne 0$, and automorphisms~$\xi_\lambda$ connect these orbits. See also~\cite[Example 2.23]{Du2004} for a relation with Makar-Limanov invariant and Gizatullin surfaces. 

For the case $n=1$ we apply Corollary~\ref{susp_An_cor} and obtain that $X=\{xz = f(y)\}$ is homogeneous if and only if $f$ has no multiple roots. Indeed, $X$ is a suspension over the affine line $Y=\AA^1$ and irreducible components of the subvariety $\{f = 0\} \subseteq \AA^1$ are distinct points. 
Also, transitivity of the action of the group $\Aut(X)$ on~$X$ can be extracted from the description of the automorphism group given in~\cite{ML1990}; see also~\cite[Example~2.23]{Du2004} and~\cite[Example~2.3]{AFKKZ12013} for different proofs of transitivity on an open subset of~$X$. 

(ii) Any homogeneous space is a homogeneous variety, so we consider only $n=1$ and the polynomial $f$ without multiple roots. 

Let us show that the surface $X_1$ in item~\ref{surf_g} from the above list is isomorphic to the Danielewski surface given by the equation
\[xz = y^2 - 1.\]
Since $X_1$ is the homogeneous space $\SL_2/T$, the algebra of regular functions $\KK[X_1]$ equals the algebra of invariants $\KK[\SL_2]^T$ with respect to the action of the torus~$T$ on $\SL_2$ by right multiplication. Since
\[\begin{pmatrix}a & b \\c & d\end{pmatrix}\begin{pmatrix}t & 0 \\ 0 & t^{-1}\end{pmatrix} = 
\begin{pmatrix}at & bt^{-1} \\ ct & dt^{-1}\end{pmatrix},\]
we see that the algebra $\KK[\SL_2]^T$ is generated by monomials $ab, ad, bc, cd$. Denote 
\[x = ab, \quad z = cd, \quad y = \frac{ad + bc}{2}.\]
Taking into account the equation $ad - bc = 1$, we obtain \[\KK[\SL_2]^T = \KK[ab, ad, bc, cd] = \KK[x, y+1, y-1, z].\] So, the algebra of regular functions of the surface $X_1 = \SL_2/T$ is isomorphic to the quotient algebra $\KK[x,y,z] / (xz - (y-1)(y+1))$ as claimed. Since the degree of the polynomial $y^2-1$ equals 2 and all polynomials of degree 2 are equivalent with respect to transformations~\eqref{Daniel_change_eq}, all Danielewski surfaces with $\deg f = 2$ are homogeneous spaces from item~\ref{surf_g}. 

The equation $xz = f(y)$ with the polynomial $f(y)$ of degree~$1$ is isomorphic to the affine plane $\AA^2$ in item~\ref{surf_a}, so all Danielewski surfaces with $\deg f = 1$ are homogeneous spaces. 

Notice that a smooth flexible variety~$X$ admits no nonconstant invertible regular functions since it is the image of an affine space via the orbit map $H_1 \times H_2 \times \ldots \times H_m \to X$ for some $\GG_\mathrm{a}$-subgroups $H_1, H_2, \ldots, H_m$ in $\Aut(X)$, see \cite[Proposition 1.5]{AFKKZ12013}. It follows that a homogeneous Danielewski surface $X$ is not isomorphic to the surfaces~\ref{surf_b} and~\ref{surf_v}. 

Let us show that the surface $X_2 = \SL_2/N$ in item~\ref{surf_d} is not isomorphic to any Danielewski surface. 
Since $\Pic(\SL_2) = 0$, the Picard group $\Pic(X_2)$ equals the character group $\fX(N) = \ZZ/2\ZZ$ of $N$ by~\cite[Corollary of Theorem~4]{Po1974}. It is sufficient to prove that the Picard group of any Danielewski surface $X = \{xz = f(y)\}$ equals~$\ZZ^d$, where $d = \deg f - 1$, see also~\cite[Example~3.3]{Ar2023}. 
Denote $f(y) = \alpha(y-y_0)\ldots(y-y_d)$, $\alpha,y_i \in \KK$. Let 
\[h = \frac{x}{\alpha(y-y_1)\ldots(y-y_d)} = \frac{y - y_0}{z} \in \KK(X).\]
Consider the divisor $D = \{z = 0, \, y = y_1\} \cup \ldots \cup \{z = 0, \, y = y_d\}$. Notice that $h$ is regular on the open subset $U = X \setminus D$, which is affine by~\cite[Lemma~3.3]{Na1962}. Moreover, $x = h\alpha(y-y_1)\ldots(y-y_d)$ and $y = y_0 + hz$ on~$U$, so the algebra $\KK[U]$ is generated by $h$ and~$z$. This implies, together with $\dim U = 2$, that $U \cong \AA^2$. It is easy to see that if an open subset~$U$ of a smooth variety $X$ is isomorphic to an affine space, then the group $\Pic(X)$ is freely generated by classes of prime divisors of $D = X \setminus U$. In our case, subvarieties $\{z = 0, \, y = y_i\} \subseteq X$, $1 \le i \le d$, are isomorphic to $\AA^1$ and whence are prime components of $D$, so $\Pic(X) = \ZZ^d$. 

To sum up, we obtain the following in the case $n=1$. If $\deg f = 1$, then $X$ is the homogeneous space~$\AA^2$. If $\deg f = 2$, then $X$ is the homogeneous space $\SL_2/T$. If $\deg f \ge 3$, then $X$ is not isomorphic to surfaces in items~\ref{surf_a}--\ref{surf_d}. Thus, the surface $X$ is a homogeneous variety but not a quasihomogeneous variety with respect to an algebraic group, in particular, it is not a homogeneous space. 
\end{proof}

\begin{example} \label{Dan_deg3_ex}
According to Theorem~\ref{Dan_theor}, the Danielewski surface $xz^n = f(y)$ is a homogeneous variety but not a homogeneous space if and only if $n=1$ and $\deg f \ge 3$. For example, 
\[\{xz = y(y+1)(y+a)\} \subseteq \AA^3,\]
$a \in \KK \setminus\{0, 1\}$, is a class of such surfaces. The surfaces with parameters $a,b$ are not isomorphic if $b \ne a, 1-a, \frac{1}{a}, \frac{a-1}{a}, \frac{a}{a-1}, \frac{1}{1-a}$, see equation~\eqref{Daniel_change_eq}. So we have infinitely many pairwise non-isomorphic examples. 
\end{example}

Let us recall that a normal affine surface $X$ is a \emph{Gizatullin surface} if the automorphism group $\Aut(X)$ acts on $X$ with an open orbit such that the complement to this orbit is a finite set. It is proved in~\cite{Gi1971} in the smooth case and in~\cite{Du2004} in the normal case that a normal affine surface $X$ that is not isomorphic to $(\KK^{\times})^2$ is a Gizatullin surface if and only if $X$ admits a smooth compactification by a smooth zigzag $D$. The latter means that $X=Y\setminus D$, where $Y$ is a complete surface smooth along $D$ and $D$ is a linear chain of smooth rational curves with simple normal crossings.

Smooth Gizatullin surfaces were conjectured to be homogeneous varieties, see \cite[Conjecture~1]{Gi1971}. However, counter-examples were found in~\cite{Ko2015}. Danielewski surfaces of the form $xz=f(y)$ that we use in this paper are examples of Gizatullin surfaces. The question which smooth Gizatullin surfaces are homogeneous is still open. 

A smooth affine surface $X$ is a \emph{Danilov-Gizatullin surface} if $X=Y\setminus D$, where $Y$ is a complete smooth surface and $D$ is a smooth irreducible rational curve.
From~\cite{Giz} we know that every Danilov-Gizatullin surface is a homogeneous variety. 

It is shown in~\cite{GD} that every Danilov-Gizatullin surface $X$ is either $\PP^2\setminus H$, where $H$ is a line, or $\PP^2\setminus C$, where $C$ is a conic, or $\FF_a\setminus S$, where $\FF_a$ is the Hirzebruch surface and $S$ is an ample divisor in $\FF_a$. In the first case we denote $X\cong \AA^2$ by $V'$. In the second case we use notation $V''=\PP^2\setminus C$. Finally, let us say that for $n\ge 2$ the $n$th Danilov-Gizatullin surface is the affine surface $V_n :=\FF_a\setminus S$, where $S\subseteq\FF_a$ is an ample divisor in the Hirzebruch surface $\FF_a$ with the intersection index $n:=(S,S)$. The isomorphism class of $V_n$ indeed depends only on $n$; see \cite[Theorem~5.8.1]{GD} or~\cite{FKZ-1}. Moreover, for different $n$ the surfaces $V_n$ are not isomorphic, see, e.g.,~\cite[Remark~2.10.2]{FKZ-1}.

Consider the affine hypersurface $F_n:=V(T_1T_4-T_2^{n-1}T_3-1)\subseteq\AA^4$, where $n\ge 2$. It is a smooth affine factorial 3-fold with $\KK[F_n]^{\times}=\KK^{\times}$, which is the spectrum of the Cox ring of~$V_n$; see~\cite[Exercise~4.18]{ADHL}. In the framework of the Cox construction, the variety $V_n$ is obtained as the geometric quotient of $F_n$ with respect to the free action of the one-dimensional torus
$$
t\cdot z=(t^{-1}z_1,tz_2,t^{1-n}z_3,tz_4); 
$$
see~\cite[Proposition~2.1]{Don} and~\cite[Exercise~4.18]{ADHL} for details. In particular, the spectrum of the Cox ring of the surface $V_n$ has dimension three. 

We conclude that the only Danilov-Gizatullin surfaces that are homogeneous spaces are $V'$, $V''$ and $V_2$. Indeed, we already know that $V'\cong\AA^2$ and $V''\cong\SL_2/N$. It follows from the classification of smooth irreducible quasihomogeneous with respect to an algebraic group affine surfaces that the only such surface with 3-dimensional spectrum of the Cox ring is $SL_2/T$. Indeed, the spectrum of the Cox ring of $\AA^2$ is $\AA^2$ and the spectrum of the Cox ring of $\SL_2/N$ is $\SL_2/T$; see \cite[Example~4.5.1.13]{ADHL}. Finally, the varieties $\AA^1\times\KK^{\times}$ and $(\KK^{\times})^2$ have non-constant invertible regular functions and the Cox ring is not defined in this situation. It is easy to see, e.g. from the Cox construction, that the variety $V_2$ is isomorphic to $SL_2/T$. 

As a result, we obtain one more series $V_n$, $n\ge 3$, of affine homogeneous varieties that are not homogeneous spaces. 

\section{Homogeneous varieties of higher dimension}
\label{dimn_sec}

In this section we provide a wide class of homogeneous varieties of an arbitrary dimension that are not homogeneous spaces. 

Recall that a polynomial $p \in \KK[x_1, \ldots, x_n]$ is called a \emph{variable} if there exist polynomials $p_2, \ldots, p_n \in \KK[x_1, \ldots, x_n]$ such that $\KK[x_1, \ldots, x_n] = \KK[p, p_2, \ldots, p_n]$. 

\begin{theorem} \label{dimn_theor}
Let a smooth hypersurface $X \subseteq \AA^{n+1}$ be given by an equation of the form
\begin{equation} \label{dimn_eq}
uv = p_0(\yu)\ldots p_d(\yu),
\end{equation}
where $u$, $v$, and $\yu = (y_1, \ldots, y_{n-1})$ denote coordinates in $\AA^{n+1}$, $p_i \in \KK[\yu]$ are irreducible polynomials, and $p_0$ is a variable in $\KK[\yu]$. Then 
\begin{enumerate}
  \item $\Pic(X) = \ZZ^d$;
  \item if $d > n$ then $X$ is a homogeneous variety that is not a homogeneous space. 
\end{enumerate}
\end{theorem}

\begin{remark}
By Corollary~\ref{susp_An_cor} the hypersurface $X$ given by equation~\eqref{dimn_eq} is smooth if and only if $p_i \ne p_j$ for $0 \le i \ne j \le d$ and the subvarieties $\{p_i = 0\} \subseteq \AA^{n-1}$, are smooth and do not intersect pairwise. 
\end{remark}

For the proof we need the following bound on the rank of the Picard group of homogeneous spaces. 

\begin{lemma} \label{Picdimn_lem}
Let $X$ be an affine homogeneous space. Then $\rk \Pic(X) \le \dim X$. 
\end{lemma}

\begin{proof}
One can assume that $X = \widehat G/\widehat H$, where $\widehat G$ acts on~$X$ effectively. Then $\widehat G$ is an affine algebraic group according to \cite[Corollary 3.2.2]{Br2017}. Moreover, we can assume that $\widehat G$ is connected. According to \cite[Theorem~3]{Po1974} there exists a central isogeny $G \to \widehat G$ with $\Pic(G) = 0$. Thus, $X = G/H$, where $G$ is an affine algebraic group such that $\Pic(G) = 0$ and~$G$ acts on~$X$ with a finite kernel of non-effectivity. 
By \cite[Corollary of Theorem~4]{Po1974}, the condition $\Pic(G)=0$ implies that the group $\Pic(X)$ equals the character group $\fX(H)$ of the subgroup~$H$. 

Denote by $H^0$ the connected component of~$H$. Since $H^0$ is a connected affine algebraic group, we have a decomposition $H^0 = (T \cdot S) \rightthreetimes H^\mathrm{u}$, where $T$ is a torus, $S$ is a semisimple group, $H^\mathrm{u}$ is the unipotent radical of $H^0$ and $\cdot$ denotes the almost direct product, see~\cite[Section 6, Theorem 4]{OV}. Since semisimple and unipotent groups have no characters, the homomorphism of restriction $\phi\colon\fX(H^0) \to \fX(T)$ is injective. 

Let us show that the restriction homomorphism $\psi\colon\fX(H) \to \fX(H^0)$ has a finite kernel. Indeed, characters of $H$ that are trivial on $H^0$ are characters of the finite group $H/H^0$. 

So $\phi\circ\psi\colon \fX(H) \to \fX(T)$ has a finite kernel, whence $\rk \fX(H) \le \rk \fX(T) = \dim T$. Since the torus $T \subseteq G$ acts on $X$ with a finite kernel of non-effectivity, the factor of $T$ by this finite subgroup is a torus $\widehat T$ of the same dimension acting effectively on~$X$. By~\cite[Corollary~1]{De1970}, we have $\dim \widehat T \le \dim X$. Thus, $\rk \Pic(X) = \rk \fX(H) \le \dim T = \dim \widehat T \le \dim X$. 
\end{proof}

\begin{proof}[Proof of Theorem~\ref{dimn_theor}]
(a) After a change of variables in $\KK[\yu]$ we can assume that $f(\yu)=y_1p_1(\yu)\ldots p_d(\yu)$, where $\yu = (y_1, \ldots,y_{n-1})$. Let
\[h = \frac{u}{p_1(\yu)\ldots p_d(\yu)} = \frac{y_1}{v} \in \KK(X).\]
Consider the divisor $D = \{v = 0, \, p_1 = 0\} \cup \ldots \cup \{v = 0, \, p_d = 0\}$. The function $h$ is regular on the open subset $U = X \setminus D$, which is an affine variety according to~\cite[Lemma~3.3]{Na1962}. Note $u = hp_1(\yu)\ldots p_d(\yu)$ and $y_1 = hv$ on~$U$, whence the algebra $\KK[U]$ is generated by $n$ functions $h, v, y_2 \ldots, y_{n-1}$. Since $\dim U = n$ it follows that $U \cong \AA^n$. Then~$\Pic(X)$ is freely generated by classes of prime divisors of $D = X \setminus U$. Since $\{v = 0, \, p_i = 0\} \subseteq X$ is isomorphic to the irreducible variety $\{p_i = 0\} \subseteq \AA^{n-1}$, $1 \le i \le d$, we have $\Pic(X) = \ZZ^d$. 

\medskip

(b) According to Proposition~\ref{homog_prop} the variety $X = \Susp(\AA^{n-1}, p_0\ldots p_d)$ is homogeneous. If $d > n$ then $X$ is not homogeneous space by item~(a) and Lemma~\ref{Picdimn_lem}. 
\end{proof}

\begin{example}
Let $X$ be the direct product of the affine space $\AA^{n-2}$ and a Danielewski surface $Y = \{xz = f(y)\} \subseteq \AA^3$ with $\deg f = n+2$, where $f$ has no multiple roots. Then $\Pic(X) = \Pic(Y \times \AA^{n-2}) = \Pic(Y) = \ZZ^{n+1}$ according to~\cite[Chapter~3, Section 1]{Sh2013} and the proof of Theorem~\ref{Dan_theor}. Since $\dim X = n$, the variety $X$ is not a homogeneous space by Lemma~\ref{Picdimn_lem}. This example belongs to the class obtained in Theorem~\ref{dimn_theor}(b). 
\end{example}

\begin{example}
Consider the hypersurface
\[X = \{uv = (x+y^2)(xy+y^3+1)(xy+y^3+2)(xy+y^3+3)(xy+y^3+4)\} \subseteq \AA^4.\]
Denote $p_0 = x+y^2$, $p_i = xy+y^3+i$, $1 \le i \le 4$. Since $\KK[x+y^2, y, u, v] = \KK[x,y,u,v]$, the polynomial~$p_0$ is a variable. Notice that $p_i$, $1 \le i \le 4$, are not variables since $p_i-i = (x+y^2)y$ is reducible. Any curve $\{p_i = 0\} \subseteq \AA^2$ is smooth and for any $i \ne j$ subvarieties $\{p_i = 0\}$ and $\{p_j = 0\}$ do not intersect pairwise, $0 \le i,j \le 4$. So by Corollary~\ref{susp_An_cor} $X$ is smooth. According to Theorem~\ref{dimn_theor}, we have $\Pic(X) = \ZZ^4$, and $X$ is a homogeneous variety but not a homogeneous space as $4 > 3$. 
\end{example}


\end{document}